\documentclass[11pt]{amsart}
\usepackage{amssymb,amsmath,amsthm,amsrefs}
\usepackage[mathscr]{eucal}
\usepackage{fullpage}
\usepackage{enumerate}
\usepackage{verbatim}
\usepackage{graphicx,float}

\usepackage{color}

\usepackage{psfrag}

\newtheorem{theorem}[equation]{Theorem}
\newtheorem{lemma}[equation]{Lemma}

\newtheorem{corollary}[equation]{Corollary}

\newtheorem{definition}[equation]{Definition}

\theoremstyle{remark}
\newtheorem{remark}[equation]{Remark}
\newtheorem{notation}[equation]{Notation}

\numberwithin{equation}{section}

\newcommand{\R}{\mathbb{R}}

\newcommand{\Z}{\mathbb{Z}}
\newcommand{\N}{\mathbb{N}}
\newcommand{\Sph}{\mathbb{S}}
\newcommand{\T}{\mathbb{T}}

\newcommand{\stab}{\operatorname{Stab}}
\newcommand{\Grp}{\mathscr{G}} 
\newcommand{\GrpSigma}{\mathscr{G}^{\mathbf{\Sigma}}} 
\newcommand{\GrpC}{\mathscr{G}^{\mathbf{C}}}

\newcommand{\GrpCQ}{\mathscr{G}^{\mathbf{C}_Q}} 
\newcommand{\GrpCa}{\mathscr{G}^{\mathbf{C}_{axes}}} 
\newcommand{\GrpMp}{\mathscr{G}^{M+}} 
\newcommand{\GrpM}{\mathscr{G}} 
 
\newcommand{\HGrp}{\mathscr{H}} 

\newcommand{\Gsym}{\mathscr{G}_{sym}}

\newcommand{\Wcal}{\mathcal{W}}

\newcommand{\Sk}{{\mathscr{S}}}

\newcommand{\pp}{\mathsf{p}}
\newcommand{\ptt}{\mathsf{t}}
\newcommand{\pqq}{\mathsf{q}}
\newcommand{\QQ}{\mathsf{x}}

\newcommand{\rot}{\mathsf{R}}
\newcommand{\refl}{{\underline{\mathsf{R}}}}
 
\newcommand{\Om}{\Omega} 
\newcommand{\Ombold}{\mathbf{\Omega}} 
\newcommand{\Omboldu}{\underline{\mathbf{\Omega}}} 
\newcommand{\Omu}{{\underline{\Omega}}} 
\newcommand{\Abold}{\mathbf{A}}

\newcommand{\cone}{\mbox{$\times \hspace*{-0.244cm} \times$}}
\newcommand{\Span}{\operatorname{Span}}
\newcommand{\Ktilde}{\widetilde{K}}

\newcommand{\Qv}{{{Q\mspace{.8mu}\!\!\!\!{/}\,}}}

\title{The Lawson surfaces are determined by their symmetries and topology} 

\author{Nikolaos Kapouleas}
\address{Department of Mathematics, Brown University, Providence, RI 02912}
\email{nicos@math.brown.edu}

\author{David Wiygul}
\address{3503 S Norton Ave, 
           Los Angeles, CA 90018,
           USA 
          }
\email{david.wiygul@gmail.com}

\date{\today}

\begin{document}

\begin{abstract}
We prove that a closed embedded minimal surface in the round three-sphere which satisfies the symmetries of a Lawson surface and has the same genus is congruent 
to the Lawson surface. 
\end{abstract}

\maketitle

\section{Introduction}

\subsection*{The general framework and brief discussion of the results}
$\phantom{ab}$
\nopagebreak

In 1970 Lawson discovered the first closed embedded minimal surfaces $\xi_{m,k}$ in the round three-sphere $\Sph^3$ \cite{Lawson} other 
than the classically known great two-sphere and Clifford torus. 
Since then many more such minimal surfaces have been found \cite{KPS,kapouleas:yang,choe:soret,kapouleas:wiygul,wiygul2020,kapI,kapII:mcg,LDg,ketover2020}. 
It is natural then to ask uniqueness, classification, and characterization questions for closed embedded minimal surfaces in $\Sph^3$.  

In this direction, Almgren \cite{almgren} proved that the only closed immersed minimal surface in $\Sph^3$ of genus zero is the great two-sphere. 
More recently enormous progress has been made in the genus one case by the recent resolutions of the Lawson conjecture by Brendle \cite{brendle}, 
stating that the only closed embedded minimal surface of genus one in $\Sph^3$ is the Clifford torus, 
and the Willmore conjecture by Marques and Neves \cite{neves}, 
where they also characterize the Clifford torus and the great two-sphere as the only closed minimal surfaces in $\Sph^3$ of area $\le2\pi^2$. 
On the other hand, very little is currently known 
on uniqueness, classification, and characterization questions for closed embedded minimal surfaces of genus $\ge2$ in $\Sph^3$.  
In particular, 
there are many cases where different constructions produce such surfaces closely resembling each other, 
and although these surfaces are strongly expected to be identical,
there is actually no proof that they are. 

In this article we prove Theorem \ref{T:main} with the following characterization of the Lawson surfaces $\xi_{m,k}$: 
if a closed embedded minimal surface in $\Sph^3$ has the same genus as a Lawson surface and satisfies the symmetries of the Lawson surface 
(actually when $m=k$ the symmetries in an index-$2$ subgroup), then it is congruent to the Lawson surface. 
The proof of Theorem \ref{T:main} is based 
on establishing that the assumptions imply that the geometry of the surface in consideration shares 
increasingly similar features with the Lawson surface (see Section \ref{S:main}), 
and eventually reducing the result to proving the uniqueness of the Lawson discs 
(fundamental domains of the Lawson surfaces under the symmetries).  
The uniqueness of the Lawson discs is proved by generalizing the uniqueness of the Lawson discs of $\xi_{g,1}$ proved in \cite{LindexI},  
where the maximum principle in the style of \cite{schoen1983} is used in combination with detailed properties of the elementary spherical geometry of the Lawson construction. 

We hope that this will be a first step towards answering 
further uniqueness, classification, and characterization questions for closed embedded minimal surfaces of genus $\ge2$ in $\Sph^3$.  
In particular towards strengthening Theorem \ref{T:main} by weakening the symmetry assumptions, 
for example to a priori allowing ``flapping of the wings'' of the Lawson surfaces in the spirit of \cite{alm20}*{Section 4.2} as discussed in Remark \ref{R:GrpS};    
or towards proving similar characterization statements for other surfaces, 
for example for some doublings 
\cite{KPS,kapouleas:yang,wiygul2020,kapI,kapII:mcg,LDg,ketover2020} 
or some other desingularizations 
\cite{choe:soret,kapouleas:wiygul}, 
ensuring this way that certain different constructions provide identical surfaces.

\subsection*{Notation and conventions}
$\phantom{ab}$
\nopagebreak

In order to ensure that this article is self-contained there is some overlap with \cite{LindexI}. 
In particular, with a few exceptions (see \ref{R:0}), we adopt the notation of \cite{LindexI}, which we describe now.  

\begin{notation} 
\label{span} 
We denote by $\Sph^3 \subset \R^{4}$ the unit $3$-dimensional sphere. 
For any 
$A\subset \Sph^3 \subset \R^{4}$ 
we denote by $\Span(A)$ the span of $A$ as a subspace of $\R^{4}$  
and we set $\Sph(A):=\Span(A)\cap\Sph^3$. 
\qed 
\end{notation} 

Given now a vector subspace $V$ of the Euclidean space $\R^{4}$, 
we denote by $V^\perp$ its orthogonal complement in $\R^{4  }$,  
and we define the reflection in $\R^{4}$ with respect to $V$,   
$\refl_V: \R^{4} \to \R^{4} $, by 
\begin{equation} 
\label{reflV} 
\refl_V:= \Pi_V - \Pi_{V^\perp}, 
\end{equation}  
where $\Pi_V$ and $\Pi_{V^\perp}$ are the orthogonal projections of $\R^{4}$ onto $V$ and $V^\perp$ respectively. 
Alternatively 
$\refl_V: \R^{4  } \to \R^{4  } $ is the linear map which restricts to the identity on $V$ 
and minus the identity on $V^\perp$.  
Clearly the fixed point set of $\refl_V$ is $V$. 

\begin{definition}[$A^\perp$ and reflections $\refl_A$] 
\label{D:refl} 
Given any 
$A\subset  \Sph^3 \subset \R^{4  }$,  
we define $A^\perp:=\left(\,Span(A) \, \right)^\perp \cap \Sph^3 $  
and  
$\refl_A : \Sph^3 \to \Sph^3 $ to be the restriction to $\Sph^3$ of $\refl_{\Span(A)}$. 
Occasionally we will use simplified notation:  
for example for $A$ as before and $p\in\Sph^3$ we may write $\Sph(A,p)$ and $\refl_{A,p}$ instead of $\Sph(A\cup\{p\})$ and $\refl_{A\cup\{p\}}$ 
respectively. 
\qed 
\end{definition} 

Note that the set of fixed points of $\refl_A$ above is $\Sph(A)$ as in notation \ref{span},  
which is $\Sph^3$ 
or a great two-sphere  
or a great circle  
or the set of two antipodal points  
or the empty set, 
depending on the dimension of $\Span(A)$. 
Following now the notation in \cite{choe:hoppe}, we have the following. 

\begin{definition}[The cone construction] 
\label{D:cone} 
For $p,q\in\Sph^3$ which are not antipodal we denote 
the minimizing geodesic segment joining them by $\overline{pq}$. 
For $A,B\subset\Sph^3$ such that no point of $A$ is antipodal to a point of $B$ 
we define the cone of $A$ and $B$ in $\Sph^3$ by 
$$ 
A\cone B := \bigcup_{p\in A, \, q\in B} \overline{pq}. 
$$ 
If $A$ or $B$ contains only one point we write the point instead of $A$ or $B$ respectively; 
we have then $p\cone q = \overline{pq} $ 
for any 
$p,q\in\Sph^3$ which are not antipodal. 
More generally, 
given linearly independent $p_1,\cdots,p_k\in \Sph^3$, 
we define inductively for $k\ge3$ 
$
\overline{p_1\cdots p_k} := p_k \cone \overline{p_1\cdots p_{k-1} }.  
$
\qed 
\end{definition} 

\begin{definition}[Tetrahedra] 
\label{TEV} 
If $p_1,p_2,p_3,p_4\in \Sph^3$ are linearly independent, 
then 
$T:= \overline{p_1p_2p_3p_4}$ is called a (spherical) tetrahedron 
with vertices $p_1,p_2,p_3,p_4$, 
edges $\overline{p_ip_j}$ ($i,j=1,2,3,4$, $i\ne j$),  
and faces 
$\overline{p_2p_3p_4}$, $\overline{p_1p_3p_4}$, $\overline{p_1p_2p_4}$, and $\overline{p_1p_2p_3}$.
We call two edges of $T$ \emph{adjacent} if they share a vertex and  
similarly a face and an edge if the face contains the edge.
Two edges which are not adjacent are called \emph{opposite}. 
Finally we use the notation $E_T:= \bigcup_{i,j=1}^4\overline{p_ip_j}$, the union of the edges, 
and $V_T:=\{p_1,p_2,p_3,p_4\}$, the set of the vertices.   
\end{definition} 

If $\Grp$ is a group acting on a set $B$ and if $A$ is a subset of $B$,
then we refer to the subgroup 
  \begin{equation}
  \label{stab}
    \stab_{\Grp}(A):=\{ \mathbf{g} \in \Grp \; | \; \mathbf{g}A = A \}
  \end{equation}
as the \emph{stabilizer} of $A$ in $\Grp$.
For $A,B$ subsets of $\Sph^3$ or $\Abold$ a finite collection of such subsets, we will set
  \begin{equation}
  \label{Gsym}
\begin{gathered} 
    \Gsym^A:=\stab_{{O(4)}} A 
= \{ \mathbf{g} \in O(4) \; | \; \mathbf{g}A = A \}, 
\qquad 
\Gsym^{A,B}:= \Gsym^{A} \cap \Gsym^{B}, 
\\ 
    \Gsym^{\Abold} := \{ \mathbf{g} \in O(4) \; | \; \{ \mathbf{g}A : A\in\Abold\} = \Abold \}. 
\end{gathered} 
  \end{equation}

\subsection*{Acknowledgments}
This article was motivated by questions asked by Antonio Ros in discussions with NK during a visit to 
the Math Institute of the University of Granada in June 2019. 
NK would like to thank also the Math Institute for their hospitality.

\section{Tessellations and basic spherical geometry}

\subsection*{Rotations along or about great circles} 
$\phantom{ab}$
\nopagebreak

Note that by \ref{D:refl},  
$C^\perp$ is the great circle furthest from a given great circle $C$ in $\Sph^3$.  
(Note that the points of $C^\perp$ are at distance $\pi/2$ in $\Sph^3$ from $C$ and any point of $\Sph^3\setminus C^\perp$ 
is at distance $<\pi/2$ from $C$). 
Equivalently $C^\perp$ is the set
of poles of great hemispheres with equator $C$;
therefore $C$ and $C^\perp$ are linked. 
The group 
$\Gsym^{C \cup C^\perp } $
contains 
$\Gsym^{C } = \Gsym^{C^\perp } $ 
(which includes arbitrary rotation or reflection in the two circles)  
and includes also   
orthogonal transformations exchanging 
$C$ with $C^\perp$.  

\begin{definition}[Rotations $\rot_C^\phi$, $\rot^C_\phi$ and Killing fields $K_{C}$, $K^{C}$] 
\label{l:D:rot} 
Given a great circle $C\subset\Sph^3$, $\phi \in \R$,
and an orientation chosen on the totally orthogonal circle $C^\perp$,
we define the following: 
\newline
(i) 
the rotation about $C$ by angle $\phi$ 
is the element $\rot_C^\phi$ of $SO(4)$ preserving $C$ pointwise 
and rotating the totally orthogonal circle $C^\perp$ along itself by angle $\phi$ 
(in accordance with its chosen orientation);  
\newline
(ii) 
the Killing field $K_{C}$ on $\Sph^3$ 
and the normalized Killing field $\Ktilde_{C}$ on $\Sph^3 \setminus C $ 
are given by 
$\left.\phantom{\frac12}K_{C}\right|_p := \left. \frac{\partial}{\partial\phi} \right|_{\phi=0} \rot_{C}^\phi(p)$
$\forall p\in\Sph^3$ 
and 
$\left.\Ktilde_{C}\right|_p \, := \,  
\frac{ \left.K_{C}\right|_p }{ \left| \left.K_{C}\right|_p \right| } 
$ 
$\forall p\in\Sph^3\setminus C$.   

Assuming now an orientation chosen on $C$ we define the following: 
\newline
(iii) 
the rotation along $C$ by angle $\phi$ is $\rot^C_\phi := \rot_{C^\perp}^\phi$;  
\newline
(iv) 
the Killing field $K^{C}:=K_{C^\perp}$ on $\Sph^3$ 
and the normalized Killing field $\Ktilde^{C}:=\Ktilde_{C^\perp}$ on $\Sph^3\setminus C^\perp$.  
\qed 
\end{definition}

Note that $\rot^C_\phi = \rot_{C^\perp}^\phi$ in the vicinity of $C$ resembles a translation along $C$ and in the vicinity of $C^\perp$ a rotation about $C^\perp$. 
Note also that $K_{C}$ is defined to be a rotational Killing field around $C$, 
vanishing on $C$ and equal to the unit velocity on ${C}^\perp$.  

\begin{lemma}[Orbits]  
\label{orbits} 
For $K^C$ as in \ref{l:D:rot}, the orbits of $K^C$ (that is its flowlines) are planar circles (in $\R^4$)  
and $\forall\pp\in C $ each orbit intersects the closed hemisphere $C^\perp \cone \pp$ exactly once.  
Moreover the intersection (when nontrivial) is orthogonal. 
\end{lemma} 

\begin{proof} 
This is straightforward to check already in $\R^4$ with the 
hemisphere $C^\perp \cone \pp$ replaced by the half-three-plane containing $\pp$ and with boundary $\Span(C^\perp)$. 
By restricting then to $\Sph^3$ the result follows. 
\end{proof} 

This lemma allows us to define a projection which effectively identifies the space of orbits in discussion 
with a closed hemisphere: 

\begin{definition}[Projections by rotations]  
\label{Pi} 
For $C$ and $\pp$ as in \ref{orbits} 
we define the smooth map 
$\Pi^C_\pp :\Sph^3\to C^\perp \cone \pp$ by requiring $\Pi^C_\pp x$ to be the intersection of $C^\perp \cone \pp$ with the orbit of $K^C$ containing $x$,    
for any $x \in\Sph^3$.  
\end{definition}

\begin{definition}[Graphical sets]  
\label{graphical} 
A set $A\subset \Sph^3$ is called \emph{graphical with respect to $K^C$} (with $C$ as above) 
if each orbit of $K^C$ intersects $A$ at most once. 
If moreover $A$ is a submanifold and there are no orbits of $K^C$ which are tangent to $A$, 
then $A$ is called 
\emph{strongly graphical with respect to $K^C$}.   
\end{definition}

\subsection*{The geometry of totally orthogonal circles}
$\phantom{ab}$
\nopagebreak

We fix now some $C$ and $C^\perp$ as above, and orientations on both. 
We define $\forall \phi\in\R$ the points 
\begin{equation}
\label{points} 
\pp_{\phi} := \rot_{C^\perp}^{\phi}\,\pp_0\,\in \, C,
\qquad
\pp^{\phi} := \rot_{C}^{\phi}\,\pp^0\,\in \, C^{\perp},
\end{equation}
where $\pp_0,\pp^0$ are arbitrarily fixed points on $C$ and $C^{\perp}$ respectively. 
Using \ref{span} we further define $\forall \phi\in\R$ the great spheres 
\begin{equation}
\label{hemispheres} 
\Sigma^\phi :=  \,\Sph( C , \pp^\phi ) , 
\qquad  
\Sigma_\phi := \, \Sph( C^\perp , \pp_\phi),   
\end{equation}
and $\forall\phi,\phi'\in\R$ 
the great circles 
\begin{equation}
\label{circles}  
C_\phi^{\phi'} := \Sph( \, \pp_\phi, \pp^{\phi'} \, ).  
\end{equation}

\begin{definition}[Coordinates on $\R^4$]  
\label{D:coordinates} 
Given $C$ as above and points as in \ref{points}, 
we define coordinates $(x^1,x^2,x^3,x^4)$ on $\R^4\supset\Sph^3$ by choosing 
$$
\pp_0 :=(1,0,0,0), \qquad \pp_{\pi/2} :=(0,1,0,0), \qquad \pp^0 :=(0,0,1,0), \qquad \pp^{\pi/2} :=(0,0,0,1). 
$$
\end{definition} 

\begin{lemma}[Basic geometry related to $C$ and $C^\perp$]  
\label{L:obs} 
The following hold $\forall\phi,\phi', \phi_1,\phi'_1,\phi_2,\phi'_2  \in\R$.  
\\
(i) $ \pp_{\phi+\pi} = - \pp_\phi $  and 
$ \pp^{\phi+\pi} = - \pp^\phi $.   
Similarly 
$ \Sigma_{\phi+\pi} = \Sigma_\phi $  and 
$ \Sigma^{\phi+\pi} = \Sigma^\phi $.   
\\
(ii) 
$C_\phi^{\phi'} \cap C = \{ \pp_\phi , \pp_{\phi+\pi} \} $ 
and 
$C_\phi^{\phi'} \cap C^\perp  = \{ \pp^{\phi'} , \pp^{\phi'+\pi} \} $ 
with orthogonal intersections. 
Moreover 
$ 
C_\phi^{\phi'} 
=       
\overline{ \pp_\phi  \pp^{\phi'} } \, \cup \,      
\overline{ \pp^{\phi'} \pp_{\phi+\pi} } \, \cup \,      
\overline{ \pp_{\phi+\pi} \pp^{\phi'+\pi} } \, \cup \,      
\overline{ \pp^{\phi'+\pi} \pp_\phi  } 
. 
$ 
\\ 
(iii) 
$C \cone \pp^\phi$  and $C^\perp \cone \pp_\phi$  
are closed great hemispheres 
with boundary $C$ and $C^\perp$ and poles $\pp^\phi$ and $\pp_\phi$ respectively. 
\\ 
(iv) 
$\Sigma_\phi = ( C^\perp \cone \pp_\phi ) \cup ( C^\perp \cone \pp_{\phi+\pi} )$  
and 
$\Sigma^\phi = ( C \cone \pp^\phi ) \cup ( C \cone \pp^{\phi+\pi} )$.  
\\ 
(v) 
$\Sigma^\phi \cap C^\perp = \{ \pp^\phi , \pp^{\phi+\pi} \}$ 
and 
$\Sigma_\phi \cap C = \{ \pp_\phi, \pp_{\phi+\pi}\}$ 
with orthogonal intersections. 
\\ 
(vi) 
$C_\phi^{\phi'} = \Sigma_\phi \cap \Sigma^{\phi'}$  
with orthogonal intersection. 
\\
(vii) 
$\left( { C_\phi^{\phi'} } \right)^\perp = C_{\phi+\pi/2}^{\phi'+\pi/2}$. 
\\ 
(viii)  
$\Sigma^\phi \cap \Sigma^{\phi'} = C$ unless $\phi=\phi' \pmod \pi$ 
in which case 
$\Sigma^\phi = \Sigma^{\phi'}$.   
Similarly 
$\Sigma_\phi \cap \Sigma_{\phi'} = C^\perp$ unless $\phi=\phi' \pmod \pi$ 
in which case 
$\Sigma_\phi = \Sigma_{\phi'}$.   
In both cases the intersection angle is $\phi'-\phi \pmod \pi$. 
\\
(ix) 
$C_{\phi_1}^{\phi_1'} \cap C_{\phi_2}^{\phi_2'} = \emptyset$ 
unless 
$\phi_1=\phi_2 \pmod \pi$ 
or
$\phi'_1=\phi'_2 \pmod \pi$. 
If both conditions hold then  
$C_{\phi_1}^{\phi_1'} = C_{\phi_2}^{\phi_2'}$.  
If only the first condition holds then 
$C_{\phi_1}^{\phi_1'} \cap C_{\phi_2}^{\phi_2'} = 
\{ \pp_{\phi_1} , \pp_{\phi_1+\pi} \}$ 
with intersection angle equal to $\phi'_2-\phi'_1 \pmod \pi$. 
If only the second condition holds then 
$C_{\phi_1}^{\phi_1'} \cap C_{\phi_2}^{\phi_2'} = 
\{ \pp^{\phi_2} , \pp^{\phi_2+\pi} \}$ 
with intersection angle equal to $\phi_2-\phi_1 \pmod \pi$. 
\end{lemma}

\begin{proof} 
It is straightforward to verify all these statements by using the coordinates defined in 
\ref{D:coordinates}.  
\end{proof} 

\subsection*{Some tessellations by tetrahedra}   
$\phantom{ab}$
\nopagebreak

We fix $m,k\ge 2$ and we 
introduce the notation 
\begin{equation}
\label{qpoints} 
\begin{aligned}
s_{i} :=& {i\frac\pi{m}} \in \R , 
\qquad &
s^{j} :=& {j\frac\pi{k}} \in \R ,  
\\
\ptt_{i} :=& \pp_{ s_{i} } \in C , 
\qquad & 
\ptt^{j} :=& \pp^{ s^{j} } \in C^\perp ,  
\end{aligned} 
\qquad 
\forall i,j\in\frac12\Z. 
\end{equation}

\begin{remark}[Notation comparison with \cite{LindexI} ]  
\label{R:0} 
Note that our notation is consistent with the one in \cite{LindexI} but we use $\ptt$'s instead of $\pqq$'s, 
which are defined so that $\ptt_i=\pqq_{i+1/2}$ and $\ptt^j=\pqq^{j+1/2}$.  
\qed 
\end{remark} 

We define now $\forall i,j \in \frac{1}{2}\Z$ tetrahedra 
$\Om_i^j= \Om_i^j[m,k]$, $\Om_{i\pm}^{j\pm}= \Om_{i\pm}^{j\pm}[m,k]$, $\Om_{i\pm}^j= \Om_{i\pm}^j[m,k]$, 
and $\Om_i^{j\pm}= \Om_i^{j\pm}[m,k]$, by 
\begin{equation}
\label{Om} 
\begin{aligned}
\Om_i^j :=& \overline{\, \ptt_{i-1/2} \ptt_{i+1/2} \ptt^{j-1/2} \ptt^{j+1/2} \, },  \qquad & 
\Om_{i\pm}^{j\pm} :=& \overline{\, \ptt_{i} \ptt_{i\pm1/2} \ptt^{j} \ptt^{j\pm1/2} \, },  
\\
\Om_{i\pm}^j :=& \overline{\, \ptt_{i} \ptt_{i\pm 1/2} \ptt^{j-1/2} \ptt^{j+1/2} \, },  \qquad & 
\Om_i^{j\pm} :=& \overline{\, \ptt_{i-1/2} \ptt_{i+1/2} \ptt^{j} \ptt^{j\pm1/2} \, },  
\end{aligned}
\end{equation} 
and groups 
$\Grp_{\phantom{s} }^{\Om_i^j} = \Grp_{\phantom{s} }^{\Om_i^j [m,k]}$ and  
$\HGrp_{\phantom{s} }^{\Om_i^j} = \HGrp_{\phantom{s} }^{\Om_i^j [m,k]}$ by 
\begin{equation} 
\label{DgrpO} 
\Grp_{\phantom{s} }^{\Om_i^j} :=
  \{
    \, \refl_{\Sph^3}, 
    \refl_{\Sigma_{i\pi/m}} , 
    \refl_{\Sigma^{j\pi/k}} , 
    \refl_{C_{i\pi/m}^{j\pi/k} } \,
  \} 
  \simeq
  \Z_2\times \Z_2,  
\qquad 
\HGrp_{\phantom{s} }^{\Om_i^j} :=
  \{
    \, \refl_{\Sph^3}, 
    \refl_{C_{i\pi/m}^{j\pi/k} } \,
  \} 
  \simeq
  \Z_2,  
\end{equation} 
where $\refl_{\Sph^3}$ is the identity map on $\Sph^3$.  
We will call 
$\forall i,j\in\frac12\Z$ the circle
$\Sph( \ptt_{i},\ptt^{j} ) = {C_{i\pi/m}^{j\pi/k} } $ the \emph{axis} of $\Om_i^j$.

\begin{definition}[Tessellations] 
\label{tessellations} 
We define $\quad\Ombold:=\{\Om_i^j\}_{i,j\in\Z}$, $\quad\Omboldu:=\{\Om_i^j\}_{i,j\in\frac12+\Z}$, 
$\quad\Ombold_e:=\{\Om_i^j\}_{i+j\in2\Z,i\in\Z},\quad$ and $\quad\Ombold_o:=\{\Om_i^j\}_{i+j\in2\Z+1,i\in\Z}$. 
\end{definition}

Note that 
$\Ombold$ and $\Omboldu$ provide tessellations of $\Sph^3$ with $4km$ tetrahedra each.   
$\Ombold_e$ and $\Ombold_o$ (with $2km$ tetrahedra each) form a subdivision of $\Ombold$.  

\begin{lemma}[Properties of $\Om_i^j$] 
\label{Om:p} 
$\forall i,j\in\frac12\Z$ the spherical tetrahedron $\Om_i^j$ is compact and convex 
(that is $\overline{xy}\subset\Om_i^j$ $\forall x,y\in \Om_i^j$)    
and satisfies the following. 
\begin{enumerate}[(i)]
\item 
Its faces are contained in the spheres 
$\Sigma^{ (j\pm1/2)\pi /k }$ and $\Sigma_{  (i\pm1/2)\pi /m } $. 
\item 
The edge on $C$ has length $\pi/m$, the edge on $C^\perp$ has length $\pi/k$, and the remaining four edges have length $\pi/2$. 
Similarly the dihedral angle on the edge on $C$ is $\pi/k$, 
on $C^\perp$ is $\pi/m$, 
and on the other four edges $\pi/2$. 
\item 
Its symmetry group satisfies the following. 
\begin{enumerate}[(a)]
\item 
If $m\neq k$, then $\Gsym^{\Om_i^j}=\Grp_{\phantom{s} }^{\Om_i^j}$.  
\item 
If $m= k>2$, then $\Gsym^{\Om_i^j}$ is isomorphic 
to the dihedral group $D_4$ of order $8$ (the group of symmetries of a square)
and contains 
$\Grp_{\phantom{s} }^{\Om_i^j}$ as a subgroup of index $2$.  
\item 
If $m= k=2$, then $\Om_i^j$ is a regular tetrahedron and 
$\Gsym^{\Om_i^j}$ is isomorphic to the symmetric group $S_4$ on four elements 
and contains 
$\Grp_{\phantom{s} }^{\Om_i^j}$ as a subgroup of index $6$.  
\end{enumerate} 
\end{enumerate} 
\end{lemma} 

\begin{proof}
It is straightforward to check all these statements by using the definitions.
\end{proof}

\subsection*{Rotations of tetrahedra} 
$\phantom{ab}$
\nopagebreak

$\forall i,j\in\Z$ we define 
\begin{equation} 
\label{E:partialpm} 
\begin{aligned} 
\partial_+ \Om_i^j := & \overline{ \, \ptt_{ i-1/2 } \ptt^{ j-1/2 } \ptt^{ j+1/2 } \, } \cup \overline{ \, \ptt_{ i+1/2 } \ptt^{ j-1/2 } \ptt^{ j+1/2 } \, } 
\!\! & \!\! = & \, 
\{ \ptt_{ i-1/2 } , \ptt_{ i+1/2 } \} \cone \overline{ \, \ptt^{ j-1/2 } \ptt^{ j+1/2 } \, },  
\\ 
\partial_- \Om_i^j := & \overline{ \, \ptt_{ i-1/2 } \ptt_{ i+1/2 } \ptt^{ j-1/2 } \, } \cup \overline{ \, \ptt_{ i-1/2 } \ptt_{ i+1/2 } \ptt^{ j+1/2 } \, }    
\!\! & \!\! = & \, 
\overline{ \, \ptt_{ i-1/2 } \ptt_{ i+1/2 } \, } \cone \{ \ptt^{ j-1/2 } ,  \ptt^{ j+1/2 } \},   
\end{aligned} 
\end{equation} 
so that 
\begin{equation} 
\label{E:partialpm2} 
\partial\Om_i^j= \partial_+ \Om_i^j \cup \partial_- \Om_i^j 
\qquad \text{ and } \qquad 
Q_i^j= \partial_+ \Om_i^j \cap \partial_- \Om_i^j,  
\end{equation} 
where 
$Q_i^j\subset \partial \Om_i^j$ is the spherical quadrilateral given by 
\begin{equation} 
\label{D:Q} 
\begin{aligned} 
Q_i^j := & 
\overline{ \ptt_{ i -1/2 } \ptt^{ j -1/2  } }  
\cup 
\overline{ \ptt^{ j -1/2 } \ptt_{ i+1/2  } }  
\cup 
\overline{ \ptt_{ i+1/2 } \ptt^{ j+1/2  } }  
\cup 
\overline{ \ptt^{ j+1/2 } \ptt_{ i-1/2  } }  
\\ 
= & \{ \ptt_{ i-1/2 } , \ptt_{ i+1/2 } \} \cone \{ \ptt^{ j-1/2 } , \ptt^{ j+1/2 } \}    
=E_{\Omega_i^j} \backslash (C \cup C^\perp) . 
\end{aligned} 
\end{equation} 
We also use the notation for the set of vertices of $Q_i^j$ (or $\Om_i^j$) 
\begin{equation} 
\label{D:Qv} 
\Qv_i^j = \Qv_i^j [m,k] := V_{\Om_i^j} = 
\{ 
\ptt_{ i-1/2 } , 
\ptt_{ i+1/2  } ,  
\ptt^{ j-1/2  } ,  
\ptt^{ j+1/2 } 
\} 
. 
\end{equation}

\begin{lemma}[{$\Om_i^j$} and rotations along its axis] 
\label{L-alex} 
The following are true $\forall i,j\in\Z$ and any orbit $O$ of 
$K_{\widetilde{C}}$, 
where
$\widetilde{C} := ({C_{i\pi/m}^{j\pi/k} })^\perp =  C_{i\pi/m+\pi/2}^{j\pi/k+\pi/2} $. 
\\ (i) $( \rot^t_{\widetilde{C}} \Om_i^j ) \cap \Om_i^j= \emptyset$ for $t\in (-3\pi/2, -\pi/2)\cup  (\pi/2,3\pi/2)$. 
Moreover either 
$( \rot^{\pm\pi/2}_{\widetilde{C}} \Om_i^j ) \cap \Om_i^j   =\{ \ptt_{i} \}$ or  
$( \rot^{\pm\pi/2}_{\widetilde{C}} \Om_i^j ) \cap \Om_i^j   =\{ \ptt^{j} \}$
(depending on the orientation of ${C_{i\pi/m}^{j\pi/k} } $ and the sign).  
\\ (ii) 
For each $\Om_{i\pm}^{j\pm}$ either $O \cap \Om_{i\pm}^{j\pm} = O \cap \Om_{i}^{j} $  or $O \cap \Om_{i\pm}^{j\pm} = \emptyset$.  
\\ (iii) 
If $O\cap \Om_i^j \ne \emptyset$,  
then (recall \ref{E:partialpm})   
$O\cap \partial_{\pm} \Om_i^j = \{x_{\pm}\} $ for some $x_{\pm} \in \partial_{\pm} \Omega_i^j$.  
Moreover $O\cap\Om_i^j$ is 
a connected arc (possibly a single point) whose endpoints are $x_+$ and $x_-$.  
\\ (iv) 
If $O\cap Q_i^j\ne \emptyset$, then $x_+=x_-\in Q_i^j$  
and 
$O\cap\Om_i^j=\{x_+\}$.   
\\ (v) 
If $O\cap ( \Om_i^j \setminus \Qv_i^j ) \ne \emptyset$,  
then $O$ intersects each face of $\Om_i^j$ containing $x_+$ ($x_-$) transversely. 
\\ (vi)
$\Pi_i^j(\Om_i^j) \subset C_{i\pi/m+\pi/2}^{j\pi/k+\pi/2} \cone \pp_{i\pi/m} $ 
is homeomorphic to a closed disc with boundary 
$\Pi_i^j(Q_i^j) $, 
where
$\Pi_i^j:= \Pi^{C_{i\pi/m}^{j\pi/k}}_{\pp_{i\pi/m} } $  
is defined as in \ref{Pi}  
(recall also $\pp_{i\pi/m}  = \ptt_{i}$).  
\end{lemma} 

\begin{proof} 
We can
assume without loss of generality that $i=j=0$. 
To prove (i) note that
$H:=\pp^0\cone C^{\pi/2}_{\pi/2} \subset \Sigma_{\pi/2}$
and
$H':=\pp_0\cone C^{\pi/2}_{\pi/2} \subset \Sigma^{\pi/2}$ are orthogonal 
closed hemispheres with common boundary $C^{\pi/2}_{\pi/2}$, 
intersecting $C^0_0$ orthogonally at $\pp^0$ and $\pp_0$ respectively, 
and satisfying  
$\rot^{\pi/2}_{\widetilde{C}} (H')=H$. 
(with $C_0^0$ appropriately oriented). 
Moreover
$\overline{\ptt^{-1/2}\ptt^{1/2}} \subset H$
and
$\overline{\ptt_{-1/2}\ptt_{1/2}} \subset H'$
with both geodesic segments avoiding the boundary $C^{\pi/2}_{\pi/2}$. 
Since two orthogonal hyperplanes separate $\R^4$
into four convex connected components,
(i) follows easily.
Because each of the bisecting spheres $\Sigma_0$ and $\Sigma^0$
is preserved by the family $\rot^{C_0^0}_t$,
the orbits of $K^{C_0^0}$ cannot cross either sphere, proving (ii).

Before turning to the remaining items
we first show that no orbit of $K^{C_0^0}$ intersects any face of $\Omega_0^0$
tangentially, except at a vertex.
By the symmetries it suffices to prove that orbits
intersect
$\overline{\ptt^0 \ptt_{1/2}\ptt^{1/2}} \subset \Sigma_{\pi/2m}$
and
$\overline{\ptt_0 \ptt_{1/2} \ptt^{1/2}} \subset \Sigma^{\pi/2k}$
transversely (if at all)
except at
$\ptt_{1/2}$
(the orbit through which is tangential to $\Sigma_{\pi/2m}$)
and 
$\ptt^{1/2}$
(the orbit through which is tangential to $\Sigma^{\pi/2k}$).
Of course the spheres $\Sigma_{\pi/2m}$ and $\Sigma^{\pi/2k}$ are minimal surfaces
and neither contains $C_0^0$,
so the Killing field $K^{C_0^0}$
induces a nontrivial Jacobi field on each of them.
A point where an orbit meets one of these spheres tangentially is a zero
of the corresponding Jacobi field,
but we know these nontrivial Jacobi fields are simply first harmonics,
each of whose nodal sets consists of a single great circle.
Clearly the reflection $\refl_{\Sigma^{\pi/2}}$ ($\refl_{\Sigma_{\pi/2}}$) preserves
the sides of $\Sigma_{\pi/2m}$ ($\Sigma^{\pi/2k}$)
and reverses each orbit of $K^{C_0^0}$.
Thus orbits can meet $\Sigma_{\pi/2m}$ ($\Sigma^{\pi/2k}$) tangentially
only along $C_{\pi/2m}^{\pi/2}$ ($C_{\pi/2}^{\pi/2k}$),
which intersects $\Om_0^0$ only at
$\ptt_{1/2}$ ($\ptt^{1/2}$),
establishing the asserted transversality.

Next we argue that no orbit of $K^{C_0^0}$ intersects
any face of $\Omega_0^0$ at more than one point.
Again (by the symmetries) it suffices to
show that every orbit intersects each of the faces
$\overline{\ptt_{1/2} \ptt^{-1/2} \ptt^{1/2}} \subset \Sigma_{\pi/2m}$
and
$\overline{\ptt_{-1/2} \ptt_{1/2} \ptt^{1/2}} \subset \Sigma^{\pi/2k}$
at most once.
To see this first note that the orbits of $K^{C_0^0}$ in $\R^4 \supset \Sph^3$
are planar circles, so if one intersects a great $2$-sphere at more than one point,
then the intersection must be either a planar circle (the entire orbit) or a pair of points.
In the first case the $2$-sphere so intersected must contain $C_0^0$,
but neither the sphere
$\Sigma^{\pi/2k}$  
nor the sphere
$\Sigma_{\pi/2m}$  
contains $C_0^0$, and so the orbits of $K^{C_0^0}$ must meet these spheres at most twice.
However, the reflection $\refl_{\Sigma^{\pi/2}}$ ($\refl_{\Sigma_{\pi/2}}$)
preserves both $\Sigma_{\pi/2m}$
($\Sigma^{\pi/2k}$)
and each orbit (as a set) of $K^{C_0^0}$,
so that if an orbit intersects $\Sigma_{\pi/2m}$
($\Sigma^{\pi/2k}$) in two points,
these points must lie either on or on opposite sides of $\Sigma^{\pi/2}$ ($\Sigma_{\pi/2}$).
Since in fact $\Omega_0^0$ crosses neither sphere of symmetry 
	and
	$
	 \Sigma^{\pi/2}
	 \cap
	 \overline{\ptt_{1/2} \ptt^{-1/2} \ptt^{1/2}}
         =
	 \ptt_{1/2}
	$
	(
        $
	 \Sigma_{\pi/2}
	 \cap
	 \overline{\ptt_{-1/2} \ptt_{1/2} \ptt^{1/2}}
         =
	 \ptt^{1/2}
        $
        ), 
we see that any orbit meets each face at most once, as claimed.

Now we are ready to prove (iii), (iv), and (v).
By the symmetries
it suffices to consider an orbit $O$ intersecting $\Omega_{0+}^{0+}$.
By (i) $O$ is not contained in $\Omega_{0+}^{0+}$ and by (ii) $O$ can enter (or exit) $\Omega_{0+}^{0+}$ only through
$\overline{\ptt_0 \ptt_{1/2} \ptt^{1/2}}$
or
$\overline{\ptt^0 \ptt_{1/2} \ptt^{1/2}}$,
but by the preceding paragraph it intersects each at most once.
Since
$\overline{\ptt_{1/2} \ptt^{1/2}}$
lies on both these triangles,
it follows that any orbit $O$ meeting
$\overline{\ptt_{1/2} \ptt^{1/2}}$
intersects $\Omega_0^0$ at only one point.
If on the other hand $O$ misses
$\overline{\ptt_{1/2} \ptt^{1/2}}$,
then, by the transversality above,
it must intersect the interior of $\Omega_0^0$,
so in this case it must cross
$\overline{\ptt_0 \ptt_{1/2} \ptt^{1/2}}
 \cup
 \overline{\ptt^0 \ptt_{1/2} \ptt^{1/2}}
$
at least twice,
meaning, by the above, that in fact $O$ must intersect each of these triangles exactly once.
This completes the proof of (iii), (iv), and (v).

For (vi) set $\Pi:=\Pi_0^0$.
Since the quadrilateral $Q_0^0$ is itself a closed curve missing
$C_{\pi/2}^{\pi/2}=\Pi^{-1}\left(C_{\pi/2}^{\pi/2}\right)$,
its image $Q':=\Pi\left(Q_0^0\right)$ under $\Pi$
is likewise a closed curve missing $C_{\pi/2}^{\pi/2}$.
By item (iv) (and the embeddedness of $Q_0^0$)
it follows that $Q'$ is an embedded closed curve in the interior of
$C_{\pi/2}^{\pi/2} \cone \pp_0$,
so that $\left(C_{\pi/2}^{\pi/2} \cone \pp_0\right) \backslash Q'$
has two connected components, one the disc bounded by $Q'$
and the other the annulus bounded by $Q'$ and $C_{\pi/2}^{\pi/2}$.
Call the closure of the disc $D'$.
Since the hemisphere
$\Pi^{-1}\left(\overline{\pp_0 \pp_{\pi/2}}\right)
 =C_0^0 \cone \pp_{\pi/2} \subset \Sigma^0
$
intersects $Q_0^0$ only at
$\ptt_{1/2}$,
we see that the geodesic arc $\overline{\pp_0 \pp_{\pi/2}}$
intersects $Q'$ exactly once
(at $\ptt_{1/2}$),
and so we conclude that $\pp_0 \in D'$.
A second application of item (iv)
ensures that $\Pi\left(\Omega_0^0 \backslash Q_0^0\right)$
misses $Q'$,
but $\Omega_0^0 \backslash Q_0^0$ is connected and includes $\pp_0$,
so we have $\Pi\left(\Omega_0^0\right) \subset D'$.
Last, note that
$
 D''
 := 
\partial_- \Omega_0^0 = 
 \overline{\ptt_{-1/2} \ptt_{1/2} \ptt^{-1/2}}
  \cup
  \overline{\ptt_{-1/2} \ptt_{1/2} \ptt^{1/2}}
$
is a disc in $\Omega_0^0$ whose boundary is $Q_0^0$ and thereby mapped by $\Pi$
homeomorphically onto $Q'=\partial D'$.
It follows (by degree theory) that $\Pi(D'')=D'$,
and so of course $\Pi(\Omega_0^0)=D'$ as well.
\end{proof}

\section{The Lawson surfaces}

\subsection*{Definition and basic properties} 
$\phantom{ab}$
\nopagebreak

We briefly discuss now the Lawson surfaces 
$\xi_{m-1,k-1}$   
defined in \cite{Lawson}.  
Since $\xi_{k-1,m-1}$ is congruent to $\xi_{m-1,k-1}$, we could restrict to $m \geq k$.
Since taking $k=1$ with any $m$ gives the great two-sphere 
and taking $m=k=2$ gives the Clifford torus,
we will assume $m \geq 3$ and $k\ge2$,  
and so the genus of $\xi_{m-1,k-1}$---which equals $(k-1)(m-1)$---is at least two. 
In this article we denote by $M[m,k]$ the Lawson surface $\xi_{m-1,k-1}$ 
positioned as in Theorem \ref{T:lawson} relative to the coordinate system defined in \ref{D:coordinates}.  
$M[m,k]$ can be viewed then as a desingularization along $C$ of 
$\Wcal := \bigcup_{j=1}^k \Sigma^{(2j-1)\pi/2k}$,  
or along $C^\perp$ of $\Wcal_\perp := \bigcup_{i=1}^m \Sigma_{(2i-1)\pi/2m}$.  
It can be proved that $M[m,k]$ converges as a varifold to $\Wcal$ as $m\to\infty$ for fixed $k$. 

\begin{theorem}[Lawson 1970 \cite{Lawson} and for the uniqueness part \cite{LindexI}] 
\label{T:lawson}
Given integers $k \geq 2$, $m \geq 3$, 
and $\forall i,j\in\Z$,  
there is a unique compact connected minimal surface $D_i^j \subset \Om_i^j$ with 
$\partial D_i^j = Q_i^j$ (recall \eqref{Om} and \eqref{D:Q}).  
Moreover $D_i^j$ is a disc, minimizing area among such discs, 
and  
$$
M=M[m,k] :=\bigcup_{i+j\in2\Z} D_i^j
$$ 
is an embedded connected closed (so two-sided) smooth minimal surface 
of genus $(k-1)(m-1)$. 
\end{theorem} 

\begin{proof} 
The theorem except for the uniqueness part but including 
the existence of a minimizing disc $D_i^j$ is proved in \cite{Lawson}. 
Although the uniqueness is also claimed in \cite{Lawson}, 
the subsequent literature 
(for example \cite{choe:soret:2009}) does not assume uniqueness known. 
We provide now a simple proof of uniqueness. 

Suppose ${D'}_i^j$ is another connected minimal surface in $\Om_i^j$ with boundary $Q_i^j$. 
By \ref{L-alex} 
$\rot^t_{C_{i\pi/m+\pi/2}^{j\pi/k+\pi/2} } D_i^j$ cannot intersect 
${D'}_i^j$ for any $t\in(-\pi,0)\cup(0,\pi)$ because 
otherwise we can consider the $\sup$ or $\inf$ of such $t$'s which we call $t'$. 
For $t'$ then we would have tangential contact on one side in the interior. 
By the maximum principle 
\cite{schoen1983}*{Lemma 1}  
this would imply 
equality of the surfaces and the boundaries, a contradiction. 

By \ref{L-alex} 
the orbits which are close enough to $Q_i^j\setminus\Qv_i^j$ and intersect $\Om_i^j$  
also intersect $D_i^j$ and ${D'}_i^j$. 
Since there are no intersections for $t\ne0$ above, we conclude that  
$D_i^j$ and ${D'}_i^j$ agree on a neighborhood of $Q_i^j\setminus\Qv_i^j$ and therefore by analytic continuation they are identical. 
\end{proof}

\begin{corollary}[Umbilics on the Lawson surfaces]
\label{umb}
The surface $M=M[m,k]$ (as in \ref{T:lawson}) has exactly 
$2k+2m$ umbilics when $k>2$ and 
$4$ umbilics when $k=2$ (recall $m>2$). 
Moreover the umbilics are 
$\ptt^j$ for $j \in \frac12+\Z$, each of degree $m-2$, 
and when $k>2$, 
$\ptt_i$ for $i \in \frac12+\Z$, each of degree $k-2$. 
\end{corollary} 

\begin{proof}
Recall that from the construction of the Lawson surfaces it is clear that the circles in 
$\mathbf{C}_Q$ are contained in $M$ and are circles of reflectional symmetry. 
Since there are exactly $m$ such circles through each $\ptt^j$ 
and exactly $k$ such circles through each $\ptt_i$, 
it follows that the points are umbilical as claimed, and of degree at least as claimed.   
(Actually since by the maximum principle $\forall i',j'\in\Z$ we have $D_{i'}^{j'} \cap\partial\Om_{i'}^{j'} = Q_{i'}^{j'}$, the degrees are exactly as claimed, 
but this follows also from the total degree argument below.) 
By a result of Lawson \cite{Lawson}*{Proposition 1.5} 
the total degree of the umbilics on the surface is
$4g-4=4(k-1)(m-1)-4=2k(m-2)+2m(k-2)$, 
which implies that there are no more umbilics and the degrees are as claimed. 
\end{proof} 

\begin{corollary}[Symmetries of the Lawson discs] 
\label{Dsym0}
For $k,m,i,j$ as in \ref{T:lawson}  
$\Gsym^{D_i^j}= \Gsym^{\Om_i^j}$; 
hence $D_i^j$ and $M$ are symmetric with respect to   
$\refl_{\Sigma^{ j\pi /k } }$, $\refl_{\Sigma_{ i\pi /m } }$, and  
$\refl_{C_{i\pi/m}^{j\pi/k} } = \refl_{ \ptt_{i},\ptt^{j} } $. 
\end{corollary} 

\begin{proof} 
Since we assume $m\ge3$ the symmetries of $\Om_i^j$ are symmetries of $Q_i^j$. 
The symmetries of $Q_i^j$ are symmetries of $D_i^j$ by the uniqueness of $D_i^j$ discussed in \ref{T:lawson}. 
Conversely any symmetry of $D_i^j$ has to be a symmetry of its boundary $Q_i^j$ and then of its vertices, and hence of $\Om_i^j$ as well. 
Using \ref{Om:p}.iii, \ref{DgrpO}, and analytic continuation, this completes the proof. 
\end{proof}

\begin{lemma}[Graphical property and subdivisions of $D_i^j$] 
\label{Dsym}
For $k,m,i,j$ as in \ref{T:lawson} we have 
\\ 
(i)  
$D_i^j$ is graphical---with its interior strongly graphical---with respect to 
$K^{C_{i\pi/m}^{j\pi/k} } = K_{C_{i\pi/m+\pi/2}^{j\pi/k+\pi/2} }$ (recall \ref{graphical})  
and each orbit which intersects $\Om_i^j$ intersects $D_i^j$ as well. 
\\ 
(ii)  
Each of $D_{i\pm}^{j} := D_i^j \cap  \Om_{i\pm}^{j}$, $D_{i}^{j\pm} := D_i^j \cap  \Om_{i}^{j\pm}$, and $D_{i\pm}^{j\pm} := D_i^j \cap  \Om_{i\pm}^{j\pm}$ 
is homeomorphic to a closed disc.  
\end{lemma} 

\begin{proof} 
To prove (i) we first prove that $D_i^j$ is graphical.  
This follows by the same argument 
as in the second paragraph of the proof of \ref{T:lawson} but with ${D'}_i^j$ replaced by $D_i^j$.
Consider now the Jacobi field 
$\nu\cdot K_{C_{i\pi/m+\pi/2}^{j\pi/k+\pi/2} }$,  
which clearly by the graphical property and appropriate choice of $\nu$ is $\ge0$ on $D_i^j$ and hence by the maximum principle is $>0$ on the interior of $D_i^j$. 
This implies that the interior of $D_i^j$ is strongly graphical. 

Next we 
recall the projection map 
\begin{equation} 
\label{Piij}
\Pi_i^j:= \Pi^{C_{i\pi/m}^{j\pi/k}}_{\pp_{i\pi/m} }  
: \Om_i^j \to C_{i\pi/m+\pi/2}^{j\pi/k+\pi/2} \cone \pp_{i\pi/m} 
\end{equation} 
defined in \ref{L-alex}.vi.  
Let $D':=\Pi_i^j(\Om_i^j)$, 
which by \ref{L-alex}.vi is homeomorphic to a closed disc with $\partial D'= \Pi_i^j(Q_i^j)$.  
Clearly then $\Pi_i^j( D_i^j )\subset D'$.  
Since $\partial D_i^j=Q_i^j$ we have also 
$\Pi_i^j( \partial D_i^j ) = \partial D'$,  
and therefore 
$\Pi_i^j( D_i^j ) = D'$, 
which completes the proof of (i).  

Furthermore, as shown above, $D_i^j$ is graphical with respect to
$K^{C_{i\pi/m}^{j\pi/k}}$,
so the restriction $\Pi_i^j|_{D_i^j}$ is one-to-one.
We conclude that $\Pi_i^j$ takes $D_i^j$ homeomorphically onto 
$D'$. 
The proof of (ii) is then completed by the fact that $\Pi_i^j$ clearly respects the symmetries of $\Omega_i^j$.
\end{proof} 

By the definitions 
$\forall i,j\in\Z$ we have  
$M \cap  \Om_{i\pm}^{j\pm} = D_{i\pm}^{j\pm}$ 
when $i+j\in 2\Z$ 
and 
$M \cap  \Om_{i\pm}^{j\pm} = \emptyset$ 
when $i+j\in 2\Z+1$.  
By \ref{Dsym} 
each $D_{i\pm}^{j\pm}$ is an embedded minimal disc. 
To study $\partial D_{i\pm}^{j\pm}$ and its intersections with its two-spheres of symmetry we define 
the intersections of $D_i^j$ and $D_{i\pm}^{j\pm}$ with the bisecting two-spheres as follows. 
\begin{equation} 
\label{ab} 
\begin{aligned} 
\alpha_i^{j\pm} :=& \, 
D_i^j \cap \overline{ \, \ptt_{ i } \ptt^{j} \ptt^{j\pm\frac12} \, } = 
D_i^{j\pm} \cap \Sigma_{ i \frac\pi{m}  } ,   
\\
\alpha_i^{j} :=& \, 
D_i^j \cap \overline{ \, \ptt_{ i} \ptt^{ j -\frac12 } \ptt^{ j+\frac12 } \, } = 
D_i^{j} \cap \Sigma_{ i \frac\pi{m}  } = 
\alpha_i^{j - }  \cup \alpha_i^{j + } ,   
\\ 
\beta_{i\pm}^j :=& \, 
D_i^j \cap \overline{ \, \ptt_{i} \ptt_{i\pm\frac12} \ptt^{ j } \, } =  
D_{i\pm}^j \cap \Sigma^{ j \frac\pi{k}  }  ,  
\\ 
\beta_{i}^j :=& \, 
D_i^j \cap \overline{ \,  \ptt_{ i-\frac12 }  \ptt_{ i+\frac12 } \ptt^{ j } \, } =  
D_{i}^j \cap \Sigma^{ j \frac\pi{k}  }  = 
\beta_{i - }^j \cup \beta_{i + }^j  .  
\end{aligned} 
\end{equation}

\begin{lemma}[The $\alpha$ and $\beta$ curves] 
\label{Dpm} 
For $k,m,i,j$ as in \ref{T:lawson} the following hold. 
\\
(i)  
$D_i^j$ intersects $\overline{ \pp_{ i\frac\pi m }  \pp^{ j\frac\pi k } } = 
\overline{ \ptt_{ i }  \ptt^{ j  } } 
$ 
at a single point which we will call $\QQ_i^j$. 
\\ 
(ii)  
The sets $\alpha_i^{j - }$, $\alpha_i^{j + }$, $\beta_{i - }^j$, $\beta_{i + }^j$, $\alpha_i^{j}$, and $\beta_{i}^j$  
are connected curves with 
$\partial \alpha_i^{j - } = \{  \ptt^{ j -\frac12 } , \QQ_i^j \}$, 
$\partial \alpha_i^{j + } =  \{  \ptt^{ j+ \frac12 } , \QQ_i^j \}$, 
$\partial \beta_{i - }^j =  \{ \ptt_{ i -\frac12 } , \QQ_i^j \}$, 
$\partial \beta_{i + }^j =  \{ \ptt_{ i+ \frac12 } , \QQ_i^j \}$, 
$\partial \alpha_i^{j} =  \{ \ptt^{j-\frac12}, \ptt^{j+ \frac12 } \}$, 
and 
$\partial \beta_{i}^j = \{ \ptt_{i -\frac12 }, \ptt_{i+\frac12 } \}$. 
\\ 
(iii)  
$\partial D_{i\pm}^{j\pm} \, 
 = \, 
 \overline{ \pp_{ (2i \pm 1)\frac\pi{2m} } \,
   \pp^{ (2j \pm 1)\frac{\pi}{2k}  } }  
\cup 
\alpha_i^{j\pm} 
\cup 
\beta_{i\pm}^j
\, = \, \overline{ \ptt_{ i \pm \frac12 } \, \ptt^{ j \pm \frac12 } }  
\cup 
\alpha_i^{j\pm} 
\cup 
\beta_{i\pm}^j.$ 
\end{lemma} 
  
\begin{proof}
As in the previous proof we consider $\Pi_i^j$, 
which is a homeomorphism from $D_i^j$ onto $D'$ and moreover respects the symmetries of $\Om_i^j$. 
Using the various definitions it is then straightforward to complete the proof. 
\end{proof}

\subsection*{Groups of symmetries} 
$\phantom{ab}$
\nopagebreak

\begin{definition}
\label{collections} 
(i) 
Let 
$\mathbf{C}_Q:=  
\left\{ \Sph( \ptt_{i},\ptt^{j} ) = {C_{i\pi/m}^{j\pi/k} } 
\right\}_{i,j\in\frac12+\Z}$, 
a collection of $km$ great circles contained in $M$ by \ref{T:lawson},  
which are the axes of the tetrahedra in $\Omboldu$ 
and contain the edges of the tetrahedra in $\Ombold$.  
\\ 
(ii) 
Let 
$\mathbf{C}_{axes}:=  
\left\{ \Sph( \ptt_{i},\ptt^{j} ) = {C_{i\pi/m}^{j\pi/k} } 
\right\}_{i,j\in\Z}$, 
a collection of $km$ great circles not contained in $M$ by \ref{Dpm}(i),  
which are the axes of the Lawson tetrahedra $\Ombold$ 
and contain the edges of the tetrahedra in $\Omboldu$.   
\\ 
(iii) 
Let 
$\mathbf{C} := \mathbf{C}_Q  \cup  \mathbf{C}_{axes}$, 
a collection of $2km$ great circles.  
\\ 
(iv) 
Let 
$\mathbf{\Sigma} := \{\Sigma^{ j\pi /k }\}_{j\in\Z} \cup \{\Sigma_{ i\pi /m }\}_{i\in\Z} $, 
a collection of $k+m$ great two-spheres  
bisecting the tetrahedra in $\Ombold$ and containing the faces of the tetrahedra in $\Omboldu$. 
\end{definition}

Note that by \ref{T:lawson} and \ref{Dsym0} 
the great circles in $\mathbf{C}$ are circles of symmetry of $M$ and 
the great two-spheres in $\mathbf{\Sigma}$ are great two-spheres of symmetry of $M$. 
Moreover  
${\bigcup} \mathbf{C}_Q = {\bigcup}_{i,j\in\Z} Q_i^j = {\bigcup}_{i,j\in2\Z}  Q_i^j \subset M$ 
and  
$\bigcup \mathbf{C}_{axes} =  ( \bigcup_{j\in\Z} \Sigma^{ j\pi /k } ) \bigcap ( \bigcup_{i\in\Z} \Sigma_{ i\pi /m } )$.

\begin{definition}[Subgroups of $\Gsym^M$] 
\label{Dsubgroups}
(i) Let $\GrpM:=\Gsym^{M,C}$ (recall \ref{Gsym}). 
\\
(ii) 
Let 
$\GrpSigma$ be the group generated by $\{ \refl_\Sigma : \Sigma\in \mathbf{\Sigma} \}$.  
\\ 
(iii) 
Let 
$\GrpCQ$, $\GrpCa$, $\GrpC$ be the groups 
generated by 
$\{\refl_{C'} : C' \in \mathbf{C}_Q \}$, 
$\{\refl_{C'} : C' \in \mathbf{C}_{axes} \}$ 
and $\{\refl_{C'} : C' \in \mathbf{C} \}$ respectively.  
\\ 
(iv) 
Let 
$\GrpMp$ be the group generated by 
$\{ \refl_{C'}:C'\in \mathbf{C}_{axes} \} \, {\scriptstyle{\bigcup}} \, \{ \refl_{C'} \circ \refl_\Sigma : C' \in \mathbf{C}_Q, \Sigma\in \mathbf{\Sigma}  \}$. 
\end{definition}

\begin{lemma}[Symmetries of $M$] 
\label{L:GsymM}
The following hold if $m \geq 3$ and $k \geq 2$. 
\\  
(i) 
$\Gsym^M = \Gsym^{\Ombold_e} = \Gsym^{{\scriptstyle\bigcup}\Ombold_e} = \Gsym^{\Ombold_o} = \Gsym^{{\scriptstyle\bigcup}\Ombold_o} $ (recall \ref{Gsym}). 
\\
(ii) 
If $m\neq k$, then $\Gsym^{M}=\GrpM$;   
and if $m= k>2$, 
then $\GrpM$ is an index-$2$ subgroup of $\Gsym^{M}$ and 
$\forall\Om\in\Ombold$ 
$\GrpM\cup \Gsym^{\Om}$ generates $\Gsym^{M}$.  
\\
(iii) 
$\GrpSigma$ acts simply transitively on $\Omboldu$ 
and $\GrpCQ$ acts simply transitively on $\Ombold_e$ (or $\Ombold_o$).  
Moreover $\forall p\in\Sph^3$ and $\forall\Omu\in\Omboldu$, $\Omu$ intersects $\GrpSigma p$ exactly once. 
\\
(iv) 
$\Gsym^M$ acts transitively on $\Ombold_e$, $\Ombold_o$, and $\Omboldu$,  
having $\Gsym^\Om$ as the stabilizer of any $\Om\in\Ombold$ 
and having an index-$2$ subgroup of $\mathscr{G}_{{sym}^{\phantom{a}}}^{\,\Omu}\!$ as the stabilizer of any $\Omu\in \Omboldu$.  
\end{lemma} 

\begin{proof} 
We have
$
 \Gsym^{\Ombold_e} = \Gsym^{\Ombold_o}
   \leq \Gsym^{\bigcup \Ombold_e} = \Gsym^{\bigcup \Ombold_o}
$
as immediate consequences of the definitions.
On the other hand any element of $\Gsym^{\bigcup \Ombold_e}$
must take edges of tetrahedra in $\Ombold$
to edges of tetrahedra in $\Ombold$,
whence it follows that $\Gsym^{\Ombold_e} = \Gsym^{\bigcup \Ombold_e}$.
For (i) it remains to prove that $\Gsym^M = \Gsym^{\Ombold_e}$.
Since $m > 2$, an element of $O(4)$ preserves $\Ombold_e$
if and only if it preserves the collection of $Q_i^j$ with $i+j$ even.
By Theorem \ref{T:lawson} we therefore have
$\Gsym^{\Ombold_e} \leq \Gsym^M$,
and to prove the reverse containment it suffices to show
that $\Gsym^M$ preserves $\mathbf{C}_Q$
(since $M \cap \Om_i^j = \emptyset$ for $i+j$ odd).
In fact, since $\bigcup \mathbf{C}_Q \subset M$
(and $k, m \geq 2$),
by the maximum principle
and the construction of $M$ in Theorem \ref{T:lawson}
it is enough to establish that $\Gsym^M$ preserves
the collection
$
 \bigcup_{\Om \in \Ombold} V_{\Om}
 =
 \{\ptt_i, \, \ptt^j\}_{i,j \in \frac{1}{2} + \Z}
$
of vertices of quadrilaterals $Q_i^j$ with $i, j \in \Z$.
Now by Corollary \ref{umb} we know that
$M$ has umbilics on $C^\perp$ and no umbilics off $C \cup C^\perp$,
but, again by the maximum principle and the construction of $M$,
we have $M \cap (C \cup C^\perp) = \bigcup_{\Om \in \Ombold} V_\Om$.
Thus $\Gsym^M$ preserves $\bigcup_{\Om \in \Ombold} V_{\Om}$,
completing the proof of (i).
Note in passing (although not needed for the proof) that these arguments fail when $m=k=2$, in which case 
$\Gsym^{\{Q_i^j\}_{i+j\in 2\Z} }  < \Gsym^{\Ombold_e}$,  
$\Gsym^{\{Q_i^j\}_{i+j\in 2\Z} }  < \Gsym^{M}$  
and 
$\Gsym^{\Ombold_e} \ne \Gsym^{M}$ hold.  

As one consequence of the preceding,
$\Gsym^M$ preserves $C \cup C^\perp$.
When $k \neq m$,
$C \cap M$ and $C^\perp \cap M$
have different cardinalities,
so any symmetry of $M$ preserves each of $C$ and $C^\perp$,
proving the first clause of (ii).
When $k=m$, note that
the stabilizer in $O(4)$ of any $\Om \in \Ombold$
preserves $\Ombold_e$
and includes symmetries exchanging $C$ and $C^\perp$,
but any element of $O(4)$ preserving $C \cup C^\perp$
either exchanges these circles or preserves each one,
completing the proof of (ii).

Since each edge of each $\Om \in \Ombold$
lies in a circle in $\mathbf{C}_Q$
and each face of each $\Omu \in \Omboldu$
lies in a sphere in $\mathbf{\Sigma}$,
the transitivities asserted in (iii) are clear.
It is also clear from the definitions
that for each $\Om \in \Ombold$
and each $\Omu \in \Omboldu$
the intersections of $\GrpCQ$ with $\Gsym^{\Om}$
and of $\GrpSigma$ with $\Gsym^{\Omu}$
are both trivial, 
establishing that both actions are simply transitive.
Next, every $p \in \Sph^3$ lies in
at least one $\underline{\Omega} \in \underline{\mathbf{\Omega}}$.
If $\mathsf{T} \in \Grp^{\mathbf{\Sigma}}$ takes
$p \in \underline{\Omega}$ to another point in $\underline{\Omega}$,
then $\mathsf{T}\underline{\Omega} = \underline{\Omega}'$
for some
$\underline{\Omega}' \in \underline{\mathbf{\Omega}}$
with
$
 \underline{\Omega} \cap \underline{\Omega}'
 \neq
 \emptyset
$:
either $\underline{\Omega}' = \underline{\Omega}$
or $\underline{\Omega}'$ and $\underline{\Omega}$
share just one face, edge, or vertex.
For any two such adjacent (or identical) tetrahedra
in $\underline{\mathbf{\Omega}}$
there is an evident element of $\Grp^{\mathbf{\Sigma}}$
taking one to the other and fixing their intersection pointwise,
but we have already observed that each
$\underline{\Omega} \in \underline{\mathbf{\Omega}}$
has trivial stabilizer in $\Grp^{\mathbf{\Sigma}}$,
completing the proof of (iii).

Turning to (iv), clearly both $\GrpCQ$ and $\GrpSigma$ preserve $\mathbf{\Omega}_e$, 
so by (i) both $\GrpCQ$ and $\GrpSigma$ are subgroups of $\Gsym^M$;
the transitivities claimed in (iv) then follow from (iii).
Finally,
the stabilizer in $O(4)$ of any $\Om \in \Ombold$
preserves $\Ombold_e$,
so agrees with the stabilizer of $\Om$ in $\Gsym^M$,
while instead $\Ombold_e$ is preserved by
exactly half of the elements of the stabilizer in $O(4)$
of any $\Omu \in \Omboldu$, ending the proof.
\end{proof} 

\begin{lemma}[$\GrpM$ and subgroups] 
\label{L:GM}
(i) $\GrpM$ has order $8km$  
and acts transitively on $\Ombold_e$, $\Ombold_o$, and $\Omboldu$,  
having $\Grp^{\Om}$ as the stabilizer of any $\Om\in\Ombold$ and $\HGrp^{\Omu}$ as the stabilizer of any $\Omu\in\Omboldu$ 
(recall \ref{DgrpO}).  
\\ 
(ii) 
$\GrpSigma=\{\gamma\in \GrpM: \gamma \text{ preserves the sides of } M\}$,  
$\GrpMp = \{\gamma\in \GrpM: \gamma \text{ preserves the orientation of } M\}$,  
\\
and $\GrpC=\GrpM\cap SO(4)$. 
\\ 
(iii) $\GrpSigma$, $\GrpC$, and $\GrpMp$ are distinct index-$2$ subgroups of $\GrpM$ and so each of them has order $4km$ and any two of them generate $\GrpM$.  
\\ 
(iv) 
$\GrpCa = \GrpSigma \cap  \GrpC = \GrpSigma  \cap \GrpMp = \GrpC \cap \GrpMp$.  
\\ 
(v) 
$\GrpCQ$ and $\GrpCa$ 
are subgroups of $\GrpC$ of index $2$ and 
$\GrpCQ\cap\GrpCa$ is a subgroup of 
$\GrpCQ$ (or $\GrpCa$) of index $2$. 
\end{lemma} 

\begin{proof} 
It is clear from the definitions that
each of $\GrpSigma$, $\GrpC$, and $\GrpMp$
preserves each of $M$ and $C$,
so, using Lemma \ref{L:GsymM},
each of these three groups is a subgroup of $\Grp$.
In particular the transitivity assertions in (i)
follow from Lemma \ref{L:GsymM}(iii).
Of course the stabilizer $\Gsym^\Om$ in $O(4)$ of any $\Om \in \Ombold$
preserves $\Ombold_e$, so is a subgroup of $\Gsym^M$,
and the subgroup of $\Gsym^\Om$ preserving $C$
is easily seen to be precisely $\Grp^\Om$.
Similarly for each $\Omu \in \Omboldu$
the subgroup of $\Gsym^{\Omu}$ preserving $M$ and $C$
is $\HGrp^{\Omu}$.
The proof of (i) is now completed by observing
that $\Ombold_e$ consists of $2km$ tetrahedra
each of whose stabilizers ($\Grp^{\Om}$ for some $\Om \in \Ombold_e$)
in $\Grp$ has order $4$
(or alternatively that $\Omboldu$ has cardinality $4km$
 and each $\HGrp^{\Omu}$ order $2$).

Likewise, again using \ref{L:GsymM}(iii),
we obtain the proof of (iv),
and we find that $\GrpC$ and $\GrpMp$,
which both act transitively on $\Ombold_e$
with stabilizer of order $2$,
have order $4km$,
completing the proof of (iii),
while $\GrpCQ$ has order $2km$,
proving the first clause of (vi),
$\GrpCQ$ and $\GrpCa$
being isomorphic.
The second clause of (vi)
then follows from the observation
that $\GrpCQ \cap \GrpCa$ is the direct product
of an order-$k$ cyclic group of rotations about $C$
with an order-$m$ cyclic group of rotations about $C^\perp$.

Next, since at each $\ptt_i$ with $i \in \frac{1}{2}+\Z$
there are $k \geq 2$ great circles on $M$ meeting $C$ orthogonally,
a normal to $M$ a each such point must point along $C$,
and likewise a normal to $M$ at each $\ptt^j$ with $j \in \frac{1}{2}+\Z$
must point along $C^\perp$.
Consequently each $\refl_{C'}$ with $C' \in \mathbf{C}_Q$
reverses the sides of $M$,
while each $\refl_{\Sigma}$ with $\Sigma \in \mathbf{\Sigma}$
preserves the sides of $M$.
Each of the three subgroups occurring as a right-hand side of (ii)
is therefore proper and so has index $2$.
Moreover it is clear that (ii) holds if each equality $=$
is replaced by containment $\leq$,
but equality then follows from (iii),
completing the proof of (ii).

By the characterization in (ii)
the intersection of any two of $\GrpSigma$, $\GrpC$, and $\GrpMp$
is the intersection of all three.
For (v) it therefore suffices to prove the first equality.
That $\GrpCa \leq \GrpSigma \cap \GrpC$ is immediate from the definitions.
By (ii) $\GrpSigma \cap \GrpC$ consists of products of even number
of reflections $\refl_{\Sigma}$ with $\Sigma \in \mathbf{\Sigma}$,
but a product of any two such reflections is indeed a rotation in $\GrpCa$,
establishing (v) and ending the proof.
\end{proof} 

\begin{remark}[Blow-ups and gluing constructions] 
\label{R:1} 
It is possible to prove that if one scales up the metric by a factor of $m^2$, 
then the limit of $M[m,k]\subset\Sph^3$ as $m\to\infty$, 
while keeping $k$ and a point of $C$ fixed, 
is a singly periodic complete embedded minimal surface $\Sk[k]\subset\R^3$, 
where $C$ converges to a straight line in $\R^3$ and $\Sph^3$ to $\R^3$ equipped with the Euclidean metric. 
Here $\Sk[k]$ is the maximally symmetric singly periodic Scherk surface \cite{Scherk} if $k=2$, which is given in Cartesian coordinates $(x_1,x_2,x_3)$ 
by the simple equation $\sinh x^1 \sinh x^2 = \sin x^3$; 
or a maximally symmetric Scherk-Karcher tower if $k>2$ \cite{Ka1}. 
This procedure can be reversed to construct Lawson surfaces of high $m$ by a gluing construction (see also \cite[Remark 6.22]{LindexI}).
Such gluing constructions can be performed also on intersecting Clifford tori as in \cite{kapouleas:wiygul}. 
Finally note that if this blow-up process is instead applied around a point of $M[m,k]$ on $\T$, 
the Clifford torus whose points are equidistant from $C$ and $C^\perp$,
where one takes $m,k\to\infty$ while the ratio $m/k$ tends to a nonzero finite limit, 
then $M$ tends to a \emph{doubly} periodic Scherk surface \cite{Scherk} with the periods spanning a plane which is the limit of $\T$. 
It should be possible to carry out gluing constructions by reversing this also. 
\qed 
\end{remark}

\section{The main Theorem} 
\label{S:main}

In this section we assume given a connected closed embedded (and hence orientable) minimal surface $S$ in $\Sph^3$. 
Recall that we assume $k \geq 2$ and $m \geq 3$. 

\begin{lemma} 
\label{L:1} 
If $\GrpSigma\subset\Gsym^S$ and $S\notin \mathbf{\Sigma}$, then the following hold $\forall \Omu\in \Omboldu$ 
(recall \ref{tessellations}, \ref{collections}, and \ref{Dsubgroups}). 
\\ 
(i) 
The intersection of $S$ with each face of $\Omu$ is non-empty 
and is a union of segments or circles which are geodesic on $S$ 
and we call \emph{sides} of $\partial(S\cap\Omu)$. 
The endpoints of the sides (if any) lie on the edges of $\Omu$ 
and we call them \emph{vertices} of $\partial(S\cap\Omu)$. 
\\
(ii) 
$S\cap V_\Omu=\emptyset$ (recall \ref{TEV}) and the intersection of $S$ with any edge of $\Omu$ is orthogonal (if nonempty). 
\\ 
(iii) 
$S\cap\Omu$ is connected. 
\end{lemma} 

\begin{proof} 
By the maximum principle $S$ intersects all great two-spheres and 
so $S\cap\Sigma\ne \emptyset$ $\forall\Sigma\in \mathbf{\Sigma}$. 
Moreover since $S$ is invariant under $\refl_\Sigma$, 
$T_pS$ is also invariant under $\refl_\Sigma$ for all $p\in S\cap\Sigma$. 
It follows that either $T_pS=T_p\Sigma$ or $T_pS$ is orthogonal to $T_p\Sigma$. 
In the former case, in the vicinity of $p$, $S$ can be described as a graph by a function over $\Sigma$.  
This function has to vanish by the symmetry. 
This implies that $S=\Sigma$, a contradiction to $S\notin \mathbf{\Sigma}$. 
We conclude therefore that $S$ intersects $\Sigma$ orthogonally. 
This implies that the intersection is transverse and therefore the intersection is an embedded curve 
with the unit normal of $S$ on this curve tangential to $\Sigma$. 
Because of the symmetry the intersection curve is geodesic on $S$. 

Each edge $E$ of a tetrahedron $\Omu\in \Omboldu$ is contained in two faces and therefore the normal to $S$ at a point of $E\cap S$ is tangential to $E$. 
A vertex then cannot be contained in $S$ because the normal to $S$ there would be tangential to all three edges containing the vertex, which is impossible. 
Now each face of a tetrahedron in $\Omboldu$ is contained in some $\Sigma\in \mathbf{\Sigma}$,  
and $\Sigma$ is tessellated by faces of tetrahedra in $\Omboldu$, 
which are acted on transitively by the stabilizer of $\Sigma$ in $\GrpSigma$. 
Since $S$ intersects $\Sigma$, $S$ has to intersect all faces on $\Sigma$.  
This concludes the proof of (i) and (ii). 

Finally, since $S$ is connected, 
any two points on $S\cap\Omu$ are connected by a path on $S$.  
Let $\gamma:[0,1]\to S$ be such a path. 
Using then \ref{L:GsymM}(iii) we define a path $\gamma':[0,1]\to \Omu\cap S$ by requesting $\forall t\in[0,1]$ that 
$\Omu\cap \GrpSigma \gamma(t) = \{ \gamma'(t)     \}$.  
$\gamma'$ is clearly continuous and (iii) follows. 
\end{proof} 

\begin{lemma} 
\label{L:2} 
If $S$ and $M$ have the same genus (which is $(k-1)(m-1)$ by \ref{T:lawson})  
and $\GrpSigma\subset\Gsym^S$, 
then $\forall \Omu\in \Omboldu$ 
\begin{equation} 
\label{E:L2} 
\sum_{p\in{S\cap E_\Omu}} (\pi-\theta_p) = 
(5-2\widetilde{g}-2\widetilde{b}-1/k-1/m)\pi,
\end{equation} 
where 
$E_\Omu$ is as in \ref{TEV}, 
$\theta_p$ denotes the dihedral angle of $\Omu$ at the edge containing $p$, 
$\widetilde{g}$ is defined to be twice the genus of $S\cap\Omu$, 
and finally $\widetilde{b}$ is the number of components of $\partial(S\cap\Omu)= S\cap\partial\Omu$.  
\end{lemma} 

\begin{proof} 
By \ref{L:GM}(iv) 
$\int_S K = 4km \int_{S\cap\Omu} K$,  
where $K$ is the Gaussian curvature of $S$. 
Combining with this the Gauss-Bonnet theorem applied to both $S$ and $S\cap\Omu$,  
and using \ref{L:1}, we conclude 
$$ 
2\pi ( 2-\widetilde{g} - \widetilde{b} ) = \frac{2\pi}{4km} \chi(S) 
 + \sum_{p\in{S\cap E_\Omu}} (\pi-\theta_p),  
$$ 
where $\chi(S)$ is the Euler characteristic of $S$. 
Since $M\cap E_\Omu$ consists of four points on edges with dihedral angles $\pi/2,\pi/2,\pi/k,\pi/m$, 
the same formula holds for $M$ (instead of $S$) and 
with $0,1$ instead of $\widetilde{g}$, $ \widetilde{b} $, 
and $(3-1/k-1/m)\pi$ instead of 
$ \sum_{p\in{S\cap E_\Omu}} (\pi-\theta_p)$   
Subtracting the formula for $M$ from the formula for $S$ we conclude the proof. 
\end{proof}

\begin{corollary} 
\label{L:3}
The assumptions in \ref{L:2} imply that $S\cap\Omu$ 
is a disc and one of the following holds. 
\\
(i) 
$\partial(S\cap\Omu)= S\cap\partial\Omu$ is a geodesic quadrilateral on $S$ with each of its four vertices on a different edge of $\Omu$ 
and with the two missed edges opposite (recall \ref{TEV}) and not on $C\cup C^\perp$. 
\\ 
(ii)
$\partial(S\cap\Omu)= S\cap\partial\Omu$ is a geodesic pentagon on $S$ with 
its vertices on four different edges of $\Omu$, 
two vertices on one edge, 
and exactly one vertex on each of the other three edges. 
The two missed edges together with the edge containing two pentagon vertices form the boundary of a face of $\Omu$. 
\end{corollary}

\begin{proof} 
Suppose $ \widetilde{g} + \widetilde{b} \ge2$. 
The right-hand side of \eqref{E:L2} is then $<\pi$. 
Since each summand on the left-hand side of \eqref{E:L2} is $\ge\pi/2$, 
it follows that there is at most one point in $S\cap E_\Omu$. 
This implies by \ref{L:1}(i) that there can be no segments in the boundary and hence 
$S\cap E_\Omu =\emptyset$, 
which implies that the left-hand side of \eqref{E:L2} is $0$. 
Since $0<1/k+1/m<1$, the right-hand side cannot vanish and we have established a contradiction to our assumption that $ \widetilde{g} + \widetilde{b} \ge2$. 

We have proved that 
$ \widetilde{g} + \widetilde{b} = 1$ 
and therefore $S\cap\Omu$ is a disc.  
By \ref{L:1} then the boundary is a geodesic polygon on $S$ and its intersection with each face of $\Omu$ is a non-empty union of geodesic segments 
whose endpoints are contained in $E_\Omu\setminus V_\Omu$ (recall \ref{TEV}).  
Since each face of $\Omu$ contains at least one side, the geodesic polygon has at least four sides and at least four vertices. 
Since the right-hand side of \eqref{E:L2} is $<3\pi$, 
by again using the fact that each summand in the left-hand side is $\ge\pi/2$,
we conclude that the geodesic polygon has at most five vertices.  
It follows that there are only two possibilities: 
(i) there are exactly four vertices with exactly one side in each face 
or (ii) 
there are five vertices 
with two sides in some face and each remaining face containing exactly one side. 

No edge of $\Omu$ can contain three polygon vertices, because this would imply that each adjacent face contains at least two polygon sides.   
Suppose now there is an edge containing two vertices.  
There cannot be polygon sides connecting these vertices on both faces adjacent to the edge 
because this would contradict the connectedness of 
the boundary $\partial(S\cap\Omu)= S\cap\partial\Omu$. 
One of the adjacent faces therefore 
contains two polygon sides, and by the above we are in case (ii) and we have a pentagon. 

Suppose now we are in case (i). 
By the above we have a quadrilateral with each of its vertices on a different edge. 
Moreover no three of the edges containing polygon vertices can share a vertex of $\Omu$ because we would have then a face with three polygon vertices on its boundary. 
Let $p_1,p_2,p_3,p_4$ be the vertices of the quadrilateral numbered so that they are consecutive along the polygon and let 
$e_1,e_2,e_3,e_4$ be the edges on which they lie respectively. 
We claim now that $e_2$ and $e_4$ are opposite. 
If they were not, since $e_1$ is also adjacent to both, $e_1,e_2,e_3$ would either share a vertex or form the boundary of a face, which leads to a contradiction. 
Similarly $e_1$ and $e_3$ are opposite.  
The remaining edges are then also opposite. 
To complete the proof of case (i) we observe that if the opposite edges on $C$ and $C^\perp$ did not contain polygon vertices, 
the left-hand side of \eqref{E:L2} would equal $2\pi$ while the right-hand side is $>2\pi$, clearly a contradiction.  

Finally suppose we are in case (ii). 
Let $F_0$ be the face which contains two pentagon sides.
Since these sides are geodesic segments on $S$ meeting the edges of $F_0$
orthogonally,
$\partial F_0$ contains exactly four polygon vertices,
with at least one edge containing exactly two of these vertices.
Let $e_1$ be such an edge, calling the other two edges $e_2$ and $e_3$,
and for each $i \in \{1,2,3\}$ let $F_i$ be the face that is adjacent
to $e_i$ and is not $F_0$.
The remaining vertex of the pentagon, that is not on $\partial F_0$,
cannot be on $\partial F_1$ either, because otherwise $\partial F_1$
would contain at least three vertices but $F_1$ only one edge,
an impossibility.
The fifth pentagon vertex therefore lies on the edge $e_1'$
opposite $e_1$.
There must be two pentagon sides connecting this vertex
to the two not on $e_1$, but each of the faces $F_2$ and $F_3$
contains exactly one side, so each of $e_2$, $e_3$
contains exactly one vertex.
\end{proof}

At the moment we need to strengthen the hypothesis $\GrpSigma\subset\Gsym^S$ we have imposed so far in order to complete the proof of our theorem. 
As discussed in remark \ref{R:GrpS} it would be highly desirable to prove the theorem under the weaker hypothesis $\GrpSigma\subset\Gsym^S$ we have imposed up to now. 

\begin{theorem} 
\label{T:main} 
Let $S$ 
be a closed embedded minimal surface is $\Sph^3$ 
and $M=M[m,k]$ a Lawson surface as in Theorem \ref{T:lawson}  
with $k \geq 2$, $m \geq 3$. 
If $S$ and $M$ have the same genus (which is $(k-1)(m-1)$ by \ref{T:lawson}) and $\GrpM\subset\Gsym^S$,  
then $S$ and $M$ are congruent. 
\end{theorem} 

\begin{proof}
The assumptions of \ref{L:3} hold and therefore either \ref{L:3}.(i) or \ref{L:3}.(ii) holds. 
In both cases at least one of the edges in $C$ or $C^\perp$ contains exactly one polygon vertex $q$ of 
the boundary $\partial(S\cap\Omu)= S\cap\partial\Omu$. 
Since $S$ and these edges are invariant under reflection with respect to the axis of $\Omu$, 
we conclude that $q$ lies on the axis of $\Omu$. 
The axis of $\Omu$ is then tangential or perpendicular to $T_qS$. 
Since $S$ intersects the edges orthogonally, the former holds, and the axis is contained therefore on $S$. 
This means that there is an odd number of vertices on each of the edges in $C$ or $C^\perp$. 
This is inconsistent with case (ii) and so we conclude that case (i) holds. 

To facilitate the presentation we take $\Omu=\Om^{1/2}_{1/2}\in \Omboldu$. 
Note that $\mathbf{\Sigma}$ and $\mathbf{C}$ are preserved by $\refl_{\Sigma_{\pi/2m}}$. 
By replacing then $S$ with $\refl_{\Sigma_{\pi/2m}}S$ if necessary we ensure that the geodesic quadrilateral 
$\partial(S\cap\Omu)= S\cap\partial\Omu$ has a vertex which we call $\QQ\in\overline{\ptt_0\ptt^0}$. 
The vertices then of the geodesic quadrilateral $\partial(S\cap\Omu)= S\cap\partial\Omu$ 
are $\ptt^{1/2}$, $\QQ\in \overline{\ptt_0\ptt^0}$, $\ptt_{1/2}$, and $\refl\QQ  \in \overline{\ptt_1\ptt^1}$,  
where we set $\refl:= \refl_{\ptt^{1/2}, \ptt_{1/2}}$.  
Recall that $V_\Omu=\{\ptt_0,\ptt_1,\ptt^1,\ptt^0\}$. 
We call $\alpha$ the side of $\partial(S\cap\Omu)= S\cap\partial\Omu$ contained in the face $\overline{\ptt^1 \ptt^0 \ptt_0} \subset \Sigma_0$ 
and which has endpoints $\ptt^{1/2}$ and $\QQ$; 
and $\beta$ the side contained in the face 
$\overline{\ptt_1 \ptt_0 \ptt^0} \subset \Sigma^0$ 
which has endpoints $\ptt_{1/2}$ and $\QQ$. 
The other two sides of $\partial(S\cap\Omu)= S\cap\partial\Omu$ are then 
$\refl\alpha \subset \overline{\ptt^1 \ptt^0 \ptt_1} \subset \Sigma_{\pi/m}$ with endpoints 
$\ptt_{1/2}$ and $\refl\QQ$, 
and 
$\refl\beta \subset \overline{\ptt_1 \ptt_0 \ptt^1} \subset \Sigma^{\pi/k}$ with endpoints 
$\refl\QQ $  and $\ptt^{1/2}$. 
Since 
$\overline{\ptt^{1/2} \ptt_{1/2}} \subset S\cap\Om^{1/2}_{1/2}$, 
the former subdivides the latter into two components. 
We call $D^+_+$ the closure of the component which satisfies 
$\partial D^+_+= \overline{\ptt^{1/2} \ptt_{1/2}} \cup\alpha\cup\beta$. 

To help visualize the situation note that 
$\Om^{1/2}_{1/2}= \Om^{1-}_{1-}\cup  \Om^{1-}_{0+}\cup  \Om^{0+}_{1-}\cup  \Om^{0+}_{0+}$. 
$M$ does not intersect the interior of 
$\Om^{1-}_{0+}\cup  \Om^{0+}_{1-}$. 
Recall that by \ref{Dpm} $M\cap \Om^{0+}_{0+} = D^{0+}_{0+}$ and 
$\partial D^{0+}_{0+} \, = \, \overline{ \ptt_{ \frac12 } \, \ptt^{ \frac12 } }  \cup \alpha_0^{0+} \cup \beta_{0+}^0$ 
is a geodesic triangle on $M$ with vertices $\ptt_{1/2}$, $\ptt^{1/2}$, $\QQ_0^0$.  
Note that $\alpha_0^{0+} \subset \Sigma_0$, $\beta_{0+}^0 \subset\Sigma^0$, and $\QQ_0^0\in \overline{\ptt_0\ptt^0} \subset C^0_0$. 
$D^+_+$ corresponds to $D^{0+}_{0+}$, but at this moment we only know that $D^+_+ \subset \Om^{1/2}_{1/2}$, while $D^{0+}_{0+}\subset \Om^{0+}_{0+}$.  

We define now 
(recall \ref{DgrpO})  
$D:= \bigcup \{ g D^+_+ : g\in \Grp_{\phantom{s} }^{\Om_0^0} \}  $.  
Clearly then $D \subset \bigcup\Om_{\pm1/2}^{\pm1/2} = \overline{\,\ptt_{-1}\ptt^{-1}\ptt^1\ptt_1\,}$, $D\subset S$, and $\partial D= Q_0^0$. 
By rotating $\Sigma_{-\pi/m}$ 
(which contains a face of $\overline{\,\ptt_{-1}\ptt^{-1}\ptt^1\ptt_1\,}$)      
about $C^\perp$ towards $\Sigma_{-\pi/2m}$ 
(which contains a face of $\Om_0^0$) 
we obtain a contradiction using the maximum principle under the assumption that we have a first contact with $D$ before reaching 
$\Sigma_{-\pi/2m}$.  
Arguing similarly for the other faces we prove $D\subset\Om_0^0$. 
Applying then the uniqueness part of \ref{T:lawson} we conclude that $D=D_0^0$, 
which implies $S=M$ and completes the proof. 
\end{proof} 

\begin{remark}[Flapping and Theorem \ref{T:main}] 
\label{R:GrpS} 
We expect that Theorem \ref{T:main} remains true if the hypothesis 
$\GrpM\subset\Gsym^S$ is weakened to 
$\GrpSigma\subset\Gsym^S$ (recall that by \ref{L:GM}(iii) $\GrpSigma$ is an index-$2$ subgroup of $\GrpM$).  
When $k=2$ the union $\Wcal$ of two great-spheres through $C$ can satisfy 
$\GrpSigma\subset\Gsym^\Wcal$ with the spheres having an arbitrary prescribed angle. 
More generally, if $k$ is even, the union $\Wcal$ of $k$ great-spheres through $C$ can satisfy 
$\GrpSigma\subset\Gsym^\Wcal$ with two adjacent spheres having an arbitrary prescribed angle. 
The anticipated stronger version of Theorem \ref{T:main} would imply then that the usual desingularization construction of a minimal surface from such a $\Wcal$ 
cannot be successful when the angle $\ne \pi/k$.  
This would be significant progress in answering the no-flapping questions posed in \cite{alm20}*{Section 4.2}, 
going beyond the local isolation of the Lawson surfaces established in \cite{LindexI}.   
\end{remark} 

\begin{remark}[Counterexamples to modifications of Theorem \ref{T:main}]  
\label{R:2} 
A closed embedded minimal surface which satisfies all the assumptions of \ref{T:main} except that it has a different genus  
can be constructed when 
$m,k$ are large in terms of a fixed $m/k\in\N$  
by easy modifications of the doubling constructions in \cite{kapouleas:yang,wiygul2020}.  
The surface constructed is a doubling of $\T$, 
the Clifford torus whose points are equidistant from $C$ and $C^\perp$,
with $2km$ catenoidal necks located at the points of $L$, the set of centers of $\Om\cap\T$ for each $\Om\in\Ombold_e$. 
It is also known that there are stackings of the Clifford torus \cite[Remark 1.3]{wiygul2020} which are not congruent although 
they have the same symmetry group and genus.  
\end{remark}

\bibliographystyle{amsplain}
\bibliography{paper}
\end{document}